\documentclass[12pt]{article}

\usepackage{mathpazo}   % With old-style figures and real smallcaps.
\usepackage{tikz}

\usepackage{nameref}
\usepackage{mathtools}
\usepackage{empheq}
\usepackage{comment}
\usepackage[shortlabels,inline]{enumitem}
\setlist[enumerate]{nosep}
\usepackage[colorlinks=true,
linkcolor=refkey,
urlcolor=lblue,
citecolor=red]{hyperref}
\usepackage[doc,wmm,hhb]{optional}
\usepackage{xcolor}

\usepackage{float}
\usepackage{soul}
\usepackage{graphicx}
\definecolor{labelkey}{rgb}{0,0.08,0.45}
\definecolor{refkey}{rgb}{0,0.6,0.0}
\definecolor{Brown}{rgb}{0.45,0.0,0.05}
\definecolor{lime}{rgb}{0.00,0.8,0.0}
\definecolor{lblue}{rgb}{0.5,0.5,0.99}
\definecolor{OliveGreen}{rgb}{0,0.6,0}
\definecolor{tyrianpurple}{rgb}{0.4, 0.01, 0.24}
\colorlet{hlcyan}{cyan!30}

\usepackage{stmaryrd}
\usepackage{amssymb}

\makeatletter
\def\namedlabel#1#2{\begingroup
   \def\@currentlabel{#2}%
   \label{#1}\endgroup
}
\makeatother

\newcommand{\seppthree}{\setlength{\itemsep}{-3pt}}

\usepackage[margin=0.92in,footskip=0.25in]{geometry}
\newcommand{\nnn}{\ensuremath{{n\in{\mathbb N}}}}

\newcommand{\thalb}{\ensuremath{\tfrac{1}{2}}}
\newcommand{\menge}[2]{\big\{{#1}~\big |~{#2}\big\}}

\newcommand{\fenv}[1]%
{\ensuremath{\,\overrightarrow{\operatorname{env}}_{#1}}}
\newcommand{\benv}[1]%
{\ensuremath{\,\overleftarrow{\operatorname{env}}_{#1}}}

\newcommand{\scal}[2]{\left\langle{#1},{#2}  \right\rangle}

\newcommand{\cerouno}{\ensuremath{\left]0,1\right[}}
\newcommand{\RR}{\ensuremath{\mathbb R}}

\newcommand{\Fix}{\ensuremath{\operatorname{Fix}}}
\newcommand{\Id}{\ensuremath{\operatorname{Id}}}

\newcommand{\pinf}{\ensuremath{+\infty}}

{\begin{list}{}{%
\settowidth{\labelwidth}{\textrm{#1~}}%
\setlength{\leftmargin}{\labelwidth+\labelsep}}}%requires macro calc.sty
{\end{list}}
\usepackage{amsthm}
\makeatletter% This is to fix the bold after the theorem
\def\th@plain{%
	\thm@notefont{}% same as heading font
	\itshape % body font
}
\def\th@definition{%
	\thm@notefont{}% same as heading font
	\normalfont % body font
}
\makeatother
\usepackage[capitalize,nameinlink]{cleveref}
\crefname{equation}{}{equations}
\crefname{chapter}{Appendix}{chapters}
\crefname{item}{}{items}
\crefname{enumi}{}{}
\newtheorem{theorem}{Theorem}[section]
\newtheorem{lemma}[theorem]{Lemma}

\newtheorem{corollary}[theorem]{Corollary}

\newtheorem{proposition}[theorem]{Proposition}

\newtheorem{ex}[theorem]{Example}
\newtheorem{fact}[theorem]{Fact}
\newtheorem{remark}[theorem]{Remark}

\providecommand{\RR}{\mathbb{R}}

\newcommand{\fix}{\ensuremath{\operatorname{Fix}}}

\providecommand{\Id}{\operatorname{{ Id}}}

\providecommand{\fix}{\operatorname{Fix}}

\providecommand{\Id}{\operatorname{Id}}

\providecommand{\RR}{\mathbb{R}}

\definecolor{myblue}{rgb}{0.9,0.9,0.98}

  \newcommand*\mybluebox[1]{%
    \colorbox{myblue}{\hspace{1em}#1\hspace{1em}}}

\allowdisplaybreaks % or locally if problems {\allowdisplaybreaks
\usepackage{relsize}

\begin{document}

\setlength{\abovedisplayskip}{8pt}
\setlength{\belowdisplayskip}{8pt}	
\newsavebox\myboxA
\newsavebox\myboxB
\newlength\mylenA

\newcommand*\xoverline[2][0.75]{%
    \sbox{\myboxA}{$#2$}%
    \setbox\myboxB\null% Phantom box
    \ht\myboxB=\ht\myboxA%
    \dp\myboxB=\dp\myboxA%
    \wd\myboxB=#1\wd\myboxA% Scale phantom
    \sbox\myboxB{$\overline{\copy\myboxB}$}%  Overlined phantom
    \setlength\mylenA{\the\wd\myboxA}%   calc width diff
    \addtolength\mylenA{-\the\wd\myboxB}%
    \ifdim\wd\myboxB<\wd\myboxA%
       \rlap{\hskip 0.5\mylenA\usebox\myboxB}{\usebox\myboxA}%
    \else
        \hskip -0.5\mylenA\rlap{\usebox\myboxA}{\hskip 0.5\mylenA\usebox\myboxB}%
    \fi}
\makeatother

\makeatletter
\renewcommand*\env@matrix[1][\arraystretch]{%
  \edef\arraystretch{#1}%
  \hskip -\arraycolsep
  \let\@ifnextchar\new@ifnextchar
  \array{*\c@MaxMatrixCols c}}
\makeatother

\providecommand{\wbar}{\xoverline[0.9]{w}}
\providecommand{\ubar}{\xoverline{u}}

\newcommand{\nn}[1]{\ensuremath{\textstyle\mathsmaller{({#1})}}}
\newcommand{\crefpart}[2]{%
  \hyperref[#2]{\namecref{#1}~\labelcref*{#1}~\ref*{#2}}%
}
\newcommand\bigzero{\makebox(0,0){\text{\LARGE0}}}
	
%-------------------------------------------------------------------------

%\tikzstyle{decision} = [diamond, draw, fill=blue!50]
%\tikzstyle{line} = [draw, -stealth, thick]
%\tikzstyle{elli}=[draw, ellipse, fill=red!50,minimum height=8mm, text width=5em, text centered]
%\tikzstyle{block} = [draw, rectangle, fill=blue!50, text width=8em, text centered, minimum height=15mm, node distance=10em]
%

\author{
Heinz H.\ Bauschke\thanks{
Mathematics, University
of British Columbia,
Kelowna, B.C.\ V1V~1V7, Canada. E-mail:
\texttt{heinz.bauschke@ubc.ca}.}~~~and~
Yuan Gao\thanks{
Mathematics, University
of British Columbia,
Kelowna, B.C.\ V1V~1V7, Canada. E-mail:
\texttt{ygao75@mail.ubc.ca}.}
}

\title{\textsf{
On a result by Baillon, Bruck, and Reich
}
}

\date{April 5, 2024}

\maketitle

\begin{abstract}
It is well known that the iterates of an averaged nonexpansive mapping may only 
converge weakly to fixed point. A celebrated result 
by Baillon, Bruck, and Reich from 1978 yields strong convergence in the 
presence of linearity. In this paper, we extend this result to allow for 
flexible relaxation parameters. Examples are also provided to illustrate the results.
\end{abstract}
{ 
\small
\noindent
{\bfseries 2020 Mathematics Subject Classification:}
{Primary 
47H05, %Monotone operators and generalizations
47H09;
Secondary 
47N10, 
65K05, 
90C25.
% 90C46,
% 47H14, 
% 49M27, 
% 49M29, 
% 49N15.% General operator theory
}


\section{Introduction}

Throughout, we assume that 
\begin{empheq}[box=\mybluebox]{equation}
\text{$X$ is
a real Hilbert space 
}
\end{empheq}
with inner product 
$\scal{\cdot}{\cdot}\colon X\times X\to\RR$, 
and induced norm $\|\cdot\|$. 
We also throughout assume that 
\begin{empheq}[box=\mybluebox]{equation}
\text{$R\colon X\to X$ is nonexpansive, 
}
\end{empheq}
i.e., $(\forall x\in X)(\forall y\in X)$
$\|Rx-Ry\|\leq\|x-y\|$, and with a nonempty 
fixed point set 
\begin{empheq}[box=\mybluebox]{equation}
\fix R = \menge{x\in X}{Rx=x} \neq \varnothing. 
\end{empheq}
Finding a point in $\Fix R$ is a basic task in optimization and 
variational analysis because the solutions to many optimization problems can often be understood as fixed point sets of nonexpansive mappings; see, e.g., \cite{BC2017}. 
To find a point in $\fix R$, one employs fixed point iterations.
Iterating $R$ is not guaranteed to work as the case $R=-\Id$ shows. 
However, iterating \emph{underrelaxations} of $R$ is a successful strategy as Krasnosel'ski\u\i\ \cite{Kras} and Mann \cite{Mann} demonstrated. 
Many extensions (see the recent monograph \cite{KMbook}) exist; here, we present here one that is quite flexible and based upon a parameter sequence 
\begin{empheq}[box=\mybluebox]{equation}
(\lambda_n)_\nnn \text{~in $\RR$,}
\end{empheq}
which we fix from now on. 
Given $\lambda\in\RR$, we set 
\begin{empheq}[box=\mybluebox]{equation}
\label{e:Tlambda}
T_{\lambda} := (1-\lambda)\Id + \lambda R
\end{empheq}
or $T_{\lambda,R}$ if we need to stress $R$. 
This allows us to concisely describe the following 
result:

\begin{fact}[Reich] 
\label{f:Reich}
Suppose that $\sum_\nnn (1-\lambda_n)\lambda_n=\pinf$ 
and let $x_0\in X$. 
Then the sequence generated by 
\begin{equation}
(\forall\nnn)\quad
x_{n+1} := T_{\lambda_n}x_n
\end{equation}
converges \emph{weakly} to a point in $\Fix R$.
Moreover, $x_n-Rx_n\to 0$ and $(x_n)_\nnn$
is Fej\'er monotone with respect to $\fix R$. 
\end{fact}
\begin{proof}
This is \cite[Theorem~2]{Reich79}. (See also \cite[Theorem~5.15]{BC2017}.) 
\end{proof}

In contrast, \emph{strong} convergence and identification of the limit is possible when $R$ is \emph{linear} but the parameter sequence is \emph{constant}. 

\begin{fact}[Baillon-Bruck-Reich] 
\label{f:BBR}
Suppose that $R$ is linear and that $\lambda\in\cerouno$. 
Let $x_0\in X$. 
\begin{equation}
(\forall\nnn)\quad
x_{n+1} := T_{\lambda}x_n
\end{equation}
Then
\begin{equation}
x_n \to P_{\Fix R}x_0. 
\end{equation}
\end{fact}
\begin{proof}
This is \cite[Example~5.29]{BC2017}; however, 
the main ideas of the proof are in
\cite{BBR} and \cite{BR}. 
\end{proof}

We are now ready to present our main result, which substantially generalizes \cref{f:BBR} and which we will prove in \cref{s:main}:

\begin{theorem}[main result]
\label{t:main}
Suppose that $R$ is linear and that there exists 
$\varepsilon>0$ such that 
\begin{equation}
(\forall\nnn)\quad
\varepsilon \leq \lambda_n \leq 1-\varepsilon. 
\end{equation}
Let $x_0\in X$ and generate the sequence $(x_n)_\nnn$ by 
\begin{equation}
(\forall\nnn)\quad
x_{n+1} := T_{\lambda_n}x_n. 
\end{equation}
Then 
\begin{equation}
x_n \to P_{\Fix R}x_0. 
\end{equation}
\end{theorem}

The remainder of the paper is organized as follows.
In \cref{s:main}, we provide the proof of \cref{t:main}. 
Variants of \cref{t:main} are discussed in \cref{s:var}. 
The notation we employ in this paper is fairly standard and 
follows largely \cite{BC2017}.

\section{Proof of the main result}

\label{s:main}

From now on, we additionally assume that 
\begin{empheq}[box=\mybluebox]{equation}
\text{$R$ is linear and $R\neq \Id$.}
\end{empheq}

The idea for the next result can be traced back to a paper by Gearhard and Koshy 
(see \cite[Acceleration~3.2]{GK} and also \cite{BDHP}):

\begin{proposition}
\label{p:1228}
Suppose that $x\in X\smallsetminus\Fix R$.
Set 
\begin{equation}
\lambda_x := \scal{x}{x-Rx}/\|x-Rx\|^2. 
\end{equation}
Then $\lambda_x\geq \thalb$. 
Let $\varepsilon\in \bigl]0,\thalb\bigr]$.
Then 
\begin{equation}
[\varepsilon,1-\varepsilon]
\subseteq 
[\varepsilon,2\lambda_x-\varepsilon] 
\quad\text{and}\quad 
(\forall \lambda\in[\varepsilon,2\lambda_x-\varepsilon]) \;\; 
\|T_{\lambda_x}x\|\leq \|T_\lambda x\|\leq \|T_\varepsilon x\|.
\end{equation}
\end{proposition}
\begin{proof}
Recalling \cref{e:Tlambda}, we define the quadratic $f$ by 
\begin{align*}
f(\lambda) &:= \|T_\lambda x\|^2 = \|x+\lambda(Rx-x)\|^2\\
&=\lambda^2\|x-Rx\|^2 + 2\lambda\scal{x}{Rx-x}+\|x\|^2. 
\end{align*}
Completing the square yields
\begin{equation*}
f(\lambda) = \|x-Rx\|^2\Big(\lambda - \frac{\scal{x}{x-Rx}}{\|x-Rx\|^2} \Big)^2 + \|x\|^2 - \frac{\scal{x}{x-Rx}^2}{\|x-Rx\|^2}. 
\end{equation*}
Hence the unique minimizer of $f$ is $\lambda_x$ and
$\min f(\RR) = \|x\|^2-\scal{x}{x-Rx}^2/\|x-Rx\|^2$. 
Because $R$ is nonexpansive and $0\in\Fix R$, we have 
$\|Rx\|\leq \|x\|$
$\Leftrightarrow$
$0\leq \|x\|^2-\|Rx\|^2=\scal{x+Rx}{x-Rx}=\scal{2x-(x-Rx)}{x-Rx}$
$\Leftrightarrow$
$2\scal{x}{x-Rx}\geq \|x-Rx\|^2$
$\Leftrightarrow$
$\lambda_x\geq\thalb$
as claimed. 
This yields $2\lambda_x-\varepsilon\geq 1-\varepsilon$ and so 
$[\varepsilon,1-\varepsilon]
\subseteq 
[\varepsilon,2\lambda_x-\varepsilon]$, also as claimed. 
On the other hand, 
$f'(\lambda) = 2\|x-Rx\|^2(\lambda-\lambda_x)$.
Altogether, 
$f$ is a convex quadratric, 
$f$ strictly decreases on $\left]-\infty,\lambda_x\right]$, 
$f$ strictly increases on $\left[\lambda_x,+\infty\right[$, and
$(\forall \delta\geq 0)$
$f(\lambda_x-\delta)=f(\lambda_x+\delta)$.
Finally, let $\lambda\in[\varepsilon,2\lambda_x-\varepsilon]$
and set $\delta := \lambda_x-\varepsilon \geq \thalb-\varepsilon\geq 0$.
Then $\lambda_x-\delta = \varepsilon$, 
$\lambda_x+\delta = 2\lambda_x-\varepsilon$, 
and therefore
$f(\lambda_x)\leq f(\lambda)\leq f(\varepsilon)=f(2\lambda_x-\varepsilon)$. 
\end{proof}

From now on, we set 

\begin{empheq}[box=\mybluebox]{equation}
\label{e:olam}
\overline{\lambda} := \inf_{x\in X\smallsetminus \Fix R}\frac{\scal{x}{x-Rx}}{\|x-Rx\|^2}.
\end{empheq}
If we wish to stress $R$, we also write $\overline{\lambda}_R$  for $\overline{\lambda}$. 

\begin{corollary}
\label{c:1228}
We have 
\begin{equation}
\overline{\lambda}\geq \thalb
\;\;\text{and}\;\;
\bigl(\forall \mu\in\bigl]0,\thalb\bigr]\bigr)
(\forall\lambda\in[\mu,2\overline{\lambda}-\mu])
(\forall x\in X)\quad 
\|T_\lambda x\|\leq\|T_\mu x\|. 
\end{equation}
\end{corollary}
\begin{proof}
Adopt the notation from \cref{p:1228}. 
Then $\overline{\lambda} = \inf_{x\in X\smallsetminus \Fix R}\lambda_x \geq \thalb$ by \cref{p:1228}, and also 
$\overline{\lambda}<\pinf$. 
Next, let $\mu\in\bigl]0,\thalb\bigr]$,
let $\lambda \in [\mu,2\overline{\lambda}-\mu]$, and 
let $x\in X$. 
Then $\lambda\leq 2\lambda_x-\mu$ and 
\cref{p:1228} yields
$\|T_\lambda x\|\leq \|T_\mu x\|$ as claimed. 
\end{proof}

\begin{lemma}
\label{l:1228}
Let $\lambda,\mu$ be in $\RR$.
Then 
\begin{equation}
T_\lambda T_\mu = T_\mu T_\lambda.
\end{equation}
\end{lemma}
\begin{proof}
Indeed, 
\begin{align*}
T_\lambda T_\mu 
&= \big((1-\lambda)\Id+\lambda R\big)\big(1-\mu)\Id+\mu R\big)\\
&=(1-\lambda)(1-\mu)\Id + \big((1-\lambda)\mu+\lambda(1-\mu)\big)R+
\lambda\mu R^2\\
&= 
T_\mu T_\lambda
\end{align*}
and we are done. 
\end{proof}

\begin{proposition}
\label{p:1229}
Suppose $(\mu_n)_\nnn$ is a sequence in $[0,1]$ such that 
$(\forall \nnn)$ 
$\|T_{\lambda_n}(\cdot)\| \leq \|T_{\mu_n}(\cdot)\|$. 
Then 
\begin{equation}
(\forall x_0\in X)(\forall y\in \Fix R)(\forall\nnn)\quad
\big\|T_{\lambda_n}\cdots T_{\lambda_1}T_{\lambda_0}x_0-y\big\|
\leq
\big\|T_{\mu_n}\cdots T_{\mu_1}T_{\mu_0}x_0-y\big\|. 
\end{equation}
\end{proposition}
\begin{proof}
Let $x_0\in X$, 
$y\in\Fix R$,
and $\nnn$. 
Then
\begin{align*}
\big\|T_{\lambda_n}\cdots T_{\lambda_1}T_{\lambda_0}x_0-y\big\|
&=
\big\|T_{\lambda_n}\big(T_{\lambda_{n-1}}\cdots T_{\lambda_1}T_{\lambda_0}(x_0-y)\big)\big\|
\tag{because $y\in\Fix R$}
\\
&\leq
\big\|T_{\mu_n}\big(T_{\lambda_{n-1}}\cdots T_{\lambda_1}T_{\lambda_0}(x_0-y)\big)\big\|
\tag{by assumption}
\\
&=
\big\|T_{\lambda_{n-1}}\big(T_{\lambda_{n-2}}\cdots T_{\lambda_1}T_{\lambda_0}T_{\mu_n}(x_0-y)\big)\big\|
\tag{by \cref{l:1228}}
\\
&\leq
\big\|T_{\mu_{n-1}}\big(T_{\lambda_{n-2}}\cdots T_{\lambda_1}T_{\lambda_0}T_{\mu_n}(x_0-y)\big)\big\|
\tag{by assumption}
\\
&=
\big\|T_{\lambda_{n-2}}\big(T_{\lambda_{n-3}}\cdots T_{\lambda_1}T_{\lambda_0}T_{\mu_n}T_{\mu_{n-1}}(x_0-y)\big)\big\|
\tag{by \cref{l:1228}}
\\
&\;\,\,\vdots \\
&\leq
\big\|T_{\mu_n}\cdots T_{\mu_1}T_{\mu_0}(x_0-y)\big\|
\\
&=
\big\|T_{\mu_n}\cdots T_{\mu_1}T_{\mu_0}x_0-y\big\|
\tag{because $y\in\Fix R$}
\end{align*}
as claimed. 
\end{proof}

We now restate the main result (for the reader's convenience) and prove it:

\begin{theorem}[main result]
\label{t:mainagain}
Suppose that there exists 
$\varepsilon>0$ such that 
$(\forall\nnn)$ 
$\varepsilon \leq \lambda_n \leq 1-\varepsilon$.
Let $x_0\in X$ and 
generate the sequence $(x_n)_\nnn$ by 
$(\forall\nnn)$
$x_{n+1} := T_{\lambda_n}x_n$. 
Then 
$x_n \to P_{\Fix R}x_0$. 
\end{theorem}
\begin{proof}
Applying \cref{c:1228} with $\mu=\varepsilon$ yields
$(\forall\nnn)$ 
$\|T_{\lambda_n}(\cdot)\|\leq\|T_{\varepsilon}(\cdot)\|$. 
Next, we apply \cref{p:1229} with $y=P_{\Fix R}x_0$ and 
$(\mu_n)_\nnn = (\varepsilon)_\nnn$ to deduce that 
\begin{equation}
\label{e:magain1}
(\forall\nnn)\quad
\|x_n-P_{\Fix R}x_0\| \leq \|T_\varepsilon^{n+1}x_0-P_{\Fix R}x_0\|. 
\end{equation}
On the other hand, 
\begin{equation}
\label{e:magain2}
T_\varepsilon^{n}x_0\to P_{\Fix R}x_0
\end{equation}
by \cref{f:BBR}. 
The conclusion follows by combining \cref{e:magain1} and \cref{e:magain2}. 
\end{proof}

\begin{remark}
There are numerous papers that use 
the Baillon-Bruck-Reich result (\cref{f:BBR}). 
Whenever this is the case, there is the potential to obtain a more powerful 
result by using the more general \cref{t:main} instead.
For instance, in the recent paper \cite{BSW}, 
the authors study several recent splitting methods 
applied to normal cone operators of closed linear subspaces. A key ingredient was to apply \cref{f:BBR} to deduce 
\begin{equation}
T_\lambda^n x_0 \to P_{\Fix R}x_0,
\end{equation}
where $\lambda\in\cerouno$. A closer inspection of the proofs shows that one may instead work with flexible parameters such as those of \cref{t:main} 
and one thus obtains a more general result.
\end{remark}

\section{Variants}

\label{s:var}

\subsection{Averaged mappings}

Recall that a nonexpansive mapping $S\colon X\to X$ is $\kappa$-averaged,
if $S = (1-\kappa)\Id + \kappa N$ for some nonexpansive mapping $N$ and $\kappa\in[0,1]$. The number
\begin{equation}
\kappa(S) := \min\menge{\kappa\in[0,1]}{\text{$S$ is $\kappa$-averaged}}
\end{equation}
is called the \emph{modulus of averagedness} of $S$.
If $\kappa(S)<1$, then one says that $S$ is averaged. 
Recalling \cref{e:olam}, it follows from \cite[Lemma~2.1]{BBM}
that 
\begin{equation}
\overline{\lambda}_R
= 
\inf_{x\in X\smallsetminus \Fix R}\frac{\scal{x}{x-Rx}}{\|x-Rx\|^2}
= 
\frac{1}{2\sup_{x\in X\smallsetminus\Fix R}\frac{\|x-Rx\|^2}{2\scal{x}{x-Rx}}}
= \frac{1}{2\kappa(R)}.
\end{equation}
Hence we have the equivalence
\begin{equation}
\overline{\lambda}_R>\thalb
\;\;\Leftrightarrow\;\;
\text{$R$ is \emph{averaged}.}
\end{equation}
This allows us to derive the following variant of \cref{t:mainagain}:

\begin{theorem}[main result --- averaged mapping version]
Suppose that $R$ is $\kappa$-averaged for some $\kappa\in\left]0,1\right[$. 
Suppose that $\delta>0$ and that $(\mu_n)_\nnn$ satisfies
$(\forall\nnn)$
$\delta\leq\mu_n\leq\tfrac{1}{\kappa}-\delta$. 
Given $x_0\in X$, generate $(x_n)_\nnn$ by 
$(\forall\nnn)$ 
$x_{n+1} := T_{\mu_n}x_n$.
Then $x_n\to P_{\Fix R}x_0$. 
\end{theorem}
\begin{proof}
Because 
$R$ is $\kappa$-averaged, 
the mapping
\begin{equation}
N := \frac{R-(1-\kappa)\Id}{\kappa}
\end{equation}
is nonexpansive, with $\Fix N = \Fix R$ (and $\kappa(N)=1$), 
and
\begin{equation}
T_{\lambda,N} = T_{\lambda/\kappa,R} = T_{\lambda/\kappa}.
\end{equation}
Now set $\varepsilon := \kappa\delta$ and  $(\forall\nnn)$ $\lambda_n := \kappa\mu_n$. 
Then $(\forall\nnn)$
$\varepsilon \leq \lambda_n\leq 1-\varepsilon$. 
By \cref{t:mainagain}, 
\begin{equation}
T_{\lambda_n,N}\cdots T_{\lambda_1,N}T_{\lambda_0,N}x_0\to P_{\Fix N}x_0.
\end{equation}
On the other hand, $(\forall\nnn)$ $T_{\mu_n}=T_{\lambda_n,N}$ and 
$\Fix N=\Fix R$.
Altogether, the result follows. 
\end{proof}

\subsection{Affine mappings}

In this subsection, we suppose that $b\in X$ and  

\begin{empheq}[box=\mybluebox]{equation}
\label{e:S}
S\colon X\to X\colon x\mapsto Rx+b
\;\;\text{with}\;\;
\Fix S \neq\varnothing.
\end{empheq}

Then the following can be seen easily (see also \cite[Lemma~3.2]{BLM})

\begin{fact}
\label{f:BLM}
There exists a point $a\in X$ such that $b=a-Ra$ and the following hold:
\begin{enumerate}
\item 
$\Fix S = a+\Fix R$.
\item 
\label{f:BLM2}
$(\forall x\in X)$ $P_{\Fix S}x = a+P_{\Fix R}(x-a)$.
\item 
\label{f:BLM3}
$(\forall x\in X)$ $Sx = a + R(x-a)$. 
\end{enumerate}
\end{fact}

\begin{corollary}
\label{c:BLM}
Let $x_0\in X$. 
Then for every $\nnn$, we have 
\begin{equation}
\label{e:BLM}
T_{\lambda_n,S}\cdots T_{\lambda_0,S}x_0 = 
a+ T_{\lambda_n,R}\cdots T_{\lambda_0,R}(x_0-a),
\end{equation}
where $a$ is as in \cref{f:BLM}. 
\end{corollary}
\begin{proof}
Let $x\in X$ and $\lambda\in\RR$. 
By \cref{f:BLM}\cref{f:BLM3},
\begin{subequations}
\label{e:1230a}
\begin{align}
T_{\lambda,S}x 
&= (1-\lambda)x+\lambda Sx
= (1-\lambda)x+\lambda(a+R(x-a))\\
&=
(1-\lambda)(x-a)+\lambda R(x-a)+a\\
&=a+T_{\lambda,R}(x-a).
\end{align}
\end{subequations}
We now prove \cref{e:BLM} by induction on $n$.
The base case $n=0$ is clear from \cref{e:1230a}. 
Now assume that \cref{e:BLM} holds for some $\nnn$.
Then 
\begin{align*}
T_{\lambda_{n+1},S}T_{\lambda_n,S}\cdots T_{\lambda_0,S}x_0
&=
T_{\lambda_{n+1},S}\big(T_{\lambda_n,S}\cdots T_{\lambda_0,S}x_0\big)\\
&=
a+ T_{\lambda_{n+1},R}\big(T_{\lambda_n,S}\cdots T_{\lambda_0,S}x_0-a\big)
\tag{using \cref{e:1230a}}
\\
&= 
a+T_{\lambda_{n+1},R}T_{\lambda_n,R}\cdots T_{\lambda_0,R}(x_0-a)
\tag{using \cref{e:BLM}}
\end{align*}
and we are done.
\end{proof}

We now obtain the following affine generalization of \cref{t:main}:

\begin{theorem}[main result --- more general affine version]
\label{t:mainaffine}
Suppose that there exists 
$\varepsilon>0$ such that 
\begin{equation}
(\forall\nnn)\quad
\varepsilon \leq \lambda_n \leq 1-\varepsilon. 
\end{equation}
Let $x_0\in X$ and generate the sequence $(x_n)_\nnn$ by 
\begin{equation}
\label{e:1230c}
(\forall\nnn)\quad
x_{n+1} := T_{\lambda_n,S}x_n. 
\end{equation}
Then 
\begin{equation}
x_n \to P_{\Fix S}x_0. 
\end{equation}
\end{theorem}
\begin{proof}
Let $a$ be as in \cref{f:BLM} and \cref{c:BLM}. 
By \cref{t:main}, we have 
\begin{equation}
\label{e:1230b}
T_{\lambda_n,R}\cdots T_{\lambda_0,R}(x_0-a) \to P_{\Fix R}(x_0-a). 
\end{equation}
Then 
\begin{align*}
x_{n+1}
&= 
T_{\lambda_n,S}\cdots T_{\lambda_0,S}x_0
\tag{using \cref{e:1230c}}
\\
&= 
a+ T_{\lambda_n,R}\cdots T_{\lambda_0,R}(x_0-a)
\tag{using \cref{e:BLM}}
\\
&\to a + P_{\Fix R}(x_0-a)
\tag{using \cref{e:1230b}}\\
&= P_{\Fix S}x_0
\tag{using \cref{f:BLM}\cref{f:BLM2}}
\end{align*}
as claimed. 
\end{proof}

\subsection{aBBR: an adaptive variant of Baillon-Bruck-Reich}

\cref{t:main} opens the door for the following adaptive version of the 
Baillon-Bruck-Reich result (\cref{f:BBR}), which we call 
\emph{adaptive Baillon-Bruck-Reich} or \emph{aBBR} for short: 

\begin{theorem}[aBBR]
\label{t:aBBR}
Suppose $R$ is linear, and let $\varepsilon\in\bigl]0,\thalb\bigr]$ and 
$x_0 \in X$. Given $\nnn$, generate the next iterate $x_{n+1}$ 
from the current iterate $x_n$ as follows: 
If $x_n \in \Fix R$, then stop. 
Otherwise, compute 
\begin{equation}
\lambda_{x_n} = \frac{\scal{x_n}{x_n-Rx_n}}{\|x_n-Rx_n\|^2}
\quad\text{and }\quad 
\lambda_n := \min\{\lambda_{x_n},1-\varepsilon\}, 
\end{equation}
and update 
\begin{equation}
x_{n+1} := T_{\lambda_n}x_n. 
\end{equation}
Then 
\begin{equation}
x_n \to P_{\Fix R}x_0. 
\end{equation}
\end{theorem}
\begin{proof}
This is a consequence of \cref{t:main} 
because $\lambda_n\in\bigl[\thalb,1-\varepsilon\bigr]$ by \cref{p:1228}. 
\end{proof}

We conclude this paper with the following numerical experiment. 
Suppose $R\in\RR^{2\times 2}$ is nonexpansive with $\Fix R=\{0\}$. 
Given a starting point $x_0\in\RR^2\smallsetminus\{0\}$, we know 
that both the standard Baillon-Bruck-Reich algorithm, which we abbreviate 
as BBR, as well as aBBR produces sequences that converge to $P_{\Fix R}(x_0)=0$ 
(by \cref{f:BBR} and \cref{t:aBBR}). We experimented with various instances of $R$
and found that the behaviour essentially follows two patterns which we illustrate 
in \cref{fig:1}. These plots were generated as follows: Given $R$, we randomly generated 
100 nonzero starting points $x_0$ and we counted how many iterations are need for 
BBR and aBBR to reach $\|x_n\|<\varepsilon := 10^{-6}$. For BBR, we
varied the constant $\lambda$ in the interval $[\varepsilon,1-\varepsilon]$. 
It then happens either that the optimal $\lambda$ for BBR is $1-\varepsilon$ 
in which case the performance of BBR and aBBR is very similar (although 
it takes slightly more work to compute the iterates generated by aBBR).
Or the optimal $\lambda$ is smaller than $1-\varepsilon$ in which case it really pays off 
to run aBBR. In \cref{fig:1}(a), the matrix is 
$R=\begin{psmallmatrix}0.7&0\\0&0.2\end{psmallmatrix}$ 
while for \cref{fig:1}(b), we have 
$R=\begin{psmallmatrix}-0.9&0\\0&0.5\end{psmallmatrix}$. This simple experiment suggests that it might be beneficial to run aBBR rather than straight BBR.

\begin{figure}[ht]
  \centering
  \begin{tabular}{c c c}
   \includegraphics[scale=0.37]{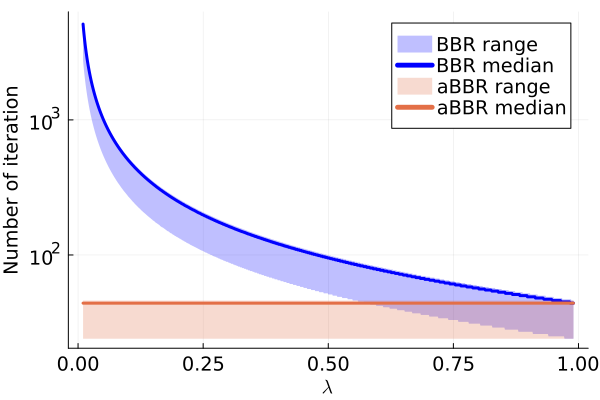}
   &
   &
   \includegraphics[scale=0.37]{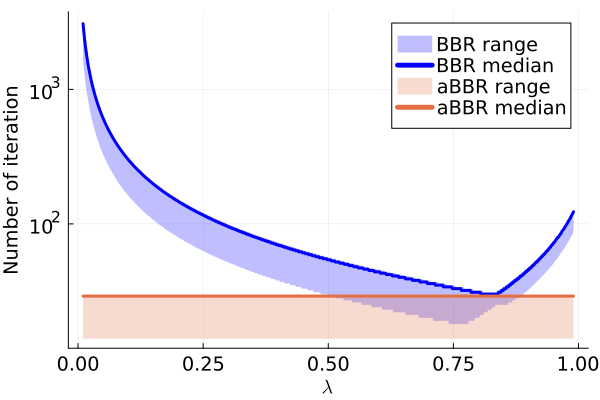}
   \\
   (a) $\lambda_\mathrm{opt} = 1-\varepsilon$
   &
   &
   (b) $\lambda_\mathrm{opt} < 1-\varepsilon$
  \end{tabular}
  \caption{Number of iterations required to achieve $\|x_n\|<\varepsilon$.}
  \label{fig:1}
\end{figure}

\section*{Acknowledgments}
\small
The research of HHB was partially supported by Discovery Grants
of the Natural Sciences and Engineering Research Council of
Canada.

\end{document}